\documentclass[11pt]{article}
\usepackage{geometry}
\usepackage{graphicx}	
\usepackage[cp1251]{inputenc}
\usepackage[english]{babel}
\usepackage{mathtools}
\usepackage{amsfonts,amssymb,mathrsfs,amscd,amsmath,amsthm}
\usepackage{verbatim}

\def\thtext#1{
  \catcode`@=11
  \gdef\@thmcountersep{. #1}
  \catcode`@=12
}

\def\threst{
  \catcode`@=11
  \gdef\@thmcountersep{.}
  \catcode`@=12
}

\theoremstyle{plain}
\newtheorem{thm}{Theorem}
\newtheorem{prop}{Proposition}[section]
\newtheorem{cor}[prop]{Corollary}
\newtheorem{ass}[prop]{Assertion}
\newtheorem{lem}[prop]{Lemma}

\theoremstyle{definition}

\newtheorem{examp}{Example}

\newtheorem{rk}[prop]{Remark}

 \catcode`@=11
 \def\.{.\spacefactor\@m}
 \catcode`@=12

\newcommand{\e}{\varepsilon}
\newcommand{\g}{\gamma}

\newcommand{\om}{\omega}

\renewcommand{\r}{\rho}

\newcommand{\cD}{\mathcal{D}}
\newcommand{\cP}{\mathcal{P}}
\newcommand{\cT}{\mathcal{T}}
\newcommand{\cTW}{\mathcal{T\!W}}

\newcommand{\N}{\mathbb{N}}

\newcommand{\R}{\mathbb{R}}

\newcommand{\rom}[1]{{\em #1}}

\newcommand{\0}{\emptyset}

\renewcommand{\c}{\circ}

\newcommand{\oPi}{\stackrel{\raise-2pt\hbox{$\c$}}\Pi}
\newcommand{\oW}{\stackrel{\raise-2pt\hbox{$\c$}}W}
\newcommand{\sm}{\setminus}
\renewcommand{\sp}{\supset}
\renewcommand{\ss}{\subset}

\renewcommand{\min}{{\operatorname{min}}}

\newcommand{\mst}{{\operatorname{mst}}}
\newcommand{\MST}{\operatorname{MST}}

\title{Minimal Spanning Trees on Infinite Sets}
\author{A.\,O.~Ivanov \and A.\,A.~Tuzhilin}
\date{}							

\begin{document}
\maketitle
\begin{abstract}
Minimal  spanning trees on infinite vertex sets are investigated. A criterion for minimality of a spanning tree having a finite length is obtained, which generalizes the corresponding classical result for finite sets. It is given an analytic description of the set of all infinite metric spaces which a minimal spanning tree exists for. A sufficient condition for a minimal spanning tree existence is obtained in terms of distances achievability between partitions elements of the metric space under consideration.  Besides, a concept of locally minimal spanning tree is introduced, several properties of such trees are described, and relations of those trees with (globally) minimal spanning trees are investigated.
\end{abstract}

\section*{Introduction}
\markright{\thesection.~Introduction}
Constructing of a minimal spanning tree connecting a given finite set of points of a metric space is a classical problem of combinatorial optimization and discrete geometry possessing many different practical applications, see, for example,~\cite{PrSh}.  From the combinatorial point of view, this problem is reduced to the problem of finding of a minimal spanning tree in a weighted graph, and in this form it was solved in the beginning of the previous century, see~\cite{Bor}.  The presence of polynomial and easy-to-describe algorithms constructing a solution (for example, Kruskal and Prim algorithms) gives almost no information concerning possible structure of the minimal spanning trees, unfortunately. Investigation of geometry of minimal spanning trees gives an opportunity to accelerate the construction algorithm in some important cases, first of all in the case of the Euclidean plane, see~\cite{Sha}.

Recently, similar objects for infinite metric spaces attract more interest. Many paper are devoted to generalizations of the concept of minimal spanning tree to the case of countable subsets of a metric space or of weighted graphs with countable sets of vertices and edges generated by some random process, see for example~\cite{AldSte} and~\cite{Alex}. In paper~\cite{AldSte}, for a countable subset $M$ of a metric space, by means of an analogue of Prim algorithm, some spanning forest is constructed (in the case of finite $M$ this forest becomes a minimal spanning tree) and some average topological and metric properties of the forest are investigated, such as an average vertex degree and an asymptotic behavior of the length. In paper~\cite{Alex} an analogue of the minimal spanning tree is defined for some natural class of weighted graphs in some other terms, namely, by means of so-called  {\it creek crossing rule\/}: an edge is chosen, if and only if its vertices can not be connected by a path consisting of edges having smaller weights. Under some natural assumptions, see~\cite{Alex},  it is shown that the graph constructed of all such edges is a forest. Also, some results concerning the structure of this forest, namely, concerning  the number of so-called topological ends and infinite clusters, are obtained.

Another good motivation for infinite minimal spanning trees studying is the Gilbert--Pollack Conjecture on the Steiner ratio of Euclidean plane, which is still open~\cite{ITOpen}, together with the fact that the best known estimate for the Steiner ratio of the Euclidean three-space is achieved at an infinite set (as a limiting value), see~\cite{SmSm}. In~\cite{INT} the following natural question is stated: Describe infinite metric spaces which are the vertices of spanning trees with finite length (in~\cite{INT} those spaces are referred as {\it good\/}). In~\cite{INT} the answer is obtained in the form of an integral criterion. In fact, the integral expression found gives the infimum of the lengths of the trees spanning the initial metric space, in particular, it gives the length of a minimal spanning tree providing such a tree does exist. But the following question has remained open: What infinite metric spaces can be spanned by minimal spanning trees?  It is easy to construct an example of a good space which can not be spanned by a minimal spanning tree. Indeed, it suffices to take the sequence $1/n$, $n\in\N$, in the real line, together with its limiting point $0$. In the present paper we get a progress in this problem solution.

It is well-known that a spanning tree on a finite subset $M$ of a metric space is minimal, if and only if the length of any its edge is equal to the distance between the corresponding components of the set  $M$. We generalize this criterion to the case of infinite subsets of a metric space (Theorem~\ref{thm:crit_mst}). Then we obtain another criterion of a spanning tree $T$ minimality, which is based on  comparison of the initial metric $\r$ with the special metric $\r^T_\infty$ constructed by $T$ and $\r$ (Theorem~\ref{thm:crit_mst_infty}). As a corollary, we give an analytic description of the set of all infinite metric spaces which a minimal spanning tree exists for (Corollary~\ref{cor:MSTctirMax1}). As another example of an application of Theorem~\ref{thm:crit_mst_infty}  we show that any tree with at most countable set of edges is isomorphic to some minimal spanning tree (Example~\ref{examp:any_tree}). Further, we obtain a sufficient condition for a minimal spanning tree existence in terms of achievability of distances between elements of partitions of the initial merit space  (Theorem~\ref{thm:distance}). In the closing Section  we introduce a concept of locally minimal spanning tree, describe some its properties and investigate relations between such trees and (globally) minimal spanning trees  (Theorem~\ref{thm:exact} and Corollary~\ref{cor:lmst=mst}).

\section{Graphs on Finite and Infinite Sets}
\markright{\thesection.~Graphs on Finite and Infinite Sets}
We consider (simple) graphs with  \textbf{arbitrary} vertex sets, not necessary finite.

Let  $V$ be an arbitrary set, $V^{(2)}$ be the family of two-elements subsets of $V$, and $E\ss V^{(2)}$. By a (simple) \emph{graph\/} we call a pair $G=(V,E)$ of such sets. As always, elements of the set  $V$ are referred as  \emph{vertices}, and elements of the set $E$ are referred as  \emph{edges\/} of the graph $G$. If the vertex set and/or the edge set of the graph $G$ are not given explicitly, then we use the notations $V(G)$ and $E(G)$ for them, respectively. Two-elements sets of the form $\{v,w\}$ we denote by $vw$. If $vw\in E$, then the vertices $v$ and $w$ are said to be \emph{adjacent\/} or \emph{neighboring\/}; we also say that those vertices are \emph{connected by an edge\/} or \emph{joined by an edge}. More generally, if $V_1,\,V_2\ss V$, $v_i\in V_i$, then we say that  \emph{$v_1v_2$ connects $V_1$ and $V_2$}. If $v\in e\in E$, then the vertex $v$ and the edge $e$ are said to be \emph{incident\/};  we also say in this case that  \emph{$v$ is a vertex of the edge $e$}, or that the \emph{edge $e$ goes out of the vertex $v$}. The cardinality of the set of edges incident to a vertex $v$ is referred as the \emph{degree of  $v$} and is denoted by  $\deg(v)$.

A finite sequence of vertices $v=v_0,v_1,\ldots,v_n=w$ such that $v_{i-1}v_i\in E$ for every  $i$ is called a \emph{route joining  $v$ and $w$}; if $v=w$, then the route is said to be \emph{closed\/} or {\em cyclic\/}; and if $v\ne w$, then the route is said to be  \emph{unclosed\/}; an unclosed route all whose vertices are pairwise distinct is called a \emph{path\/}; a closed route all whose edges are pairwise distinct is called a \emph{cycle}. If any pair of vertices of a graph is connected by some route, then the graph is said to be \emph{connected\/}; if a graph does not contain cycles, then it is called a  \emph{forest\/}; a connected forest is called a  \emph{tree}. It is easy to see, that each pair of vertices  $v$ and $w$ of an arbitrary tree $T$ is connected by a unique path. By $T[v,w]$ we denote this path in $T$.

We use several graph operations. Let $G=(V,E)$ be an arbitrary graph and $E'\ss V^{(2)}$. Then by $G\sm E'$ and $G\cup E'$ we denote the graphs $(V,E\sm E')$ and $(V,E\cup E')$, respectively; we say that the first graph is obtained from the graph $G$ by \emph{discarding the family $E'$ of edges}, and the second graph is obtained by \emph{adding the edges from the set $E'$}. If  $E'=\{e\}$, then the graphs $G\sm\{e'\}$ and $G\cup\{e'\}$ are denoted by $G\sm e$ and $G\cup e$, respectively. Further, let  $E',\,F'\ss E$, then by $G[E'\to F']$ we denote the graph $(G\sm E')\cup F'$; if $E'=\{e\}$ and  $F'=\{f\}$, then we put  $G[e\to f]=G\bigl[\{e\}\to\{f\}\bigr]$.

If $G_1=(V_1,E_1)$ and $G_2=(V_2,E_2)$, then the graph $(V_1\cup V_2,E_1\cup E_2)$ is denoted by $G_1\cup G_2$ and is called the \emph{union of the graphs $G_1$ and $G_2$}. If $V_1\cap V_2=\0$, then we write $G_1\sqcup G_2$ instead of $G_1\cup G_2$ and call the resulting graph by the \emph{disjoint union of the graphs $G_1$ and $G_2$}.

If  $G=(V,E)$ is an arbitrary graph, and $V'\ss V$, $E'\ss E$, then the graph $G'=(V',E')$ is referred as a  \emph{subgraph of the graph $G$}, and we write $G'\ss G$. The relation of being a subgraph is an order relation, and in what follows we understand  \emph{maximal\/} or  \emph{minimal subgraphs with respect to inclusion\/} in the sense of this ordering. A maximal with respect to inclusion connected subgraph of a graph $G$ is called a  \emph{connected component of the graph $G$}; each graph can be uniquely represented as disjoint union of its connected components; in particular, each forest $G$ can be uniquely decomposed into disjoint union of its maximal subtrees which are referred as the  \emph{trees of the forest $G$}. In particular, if $T=(V,E)$ is a tree, and $E'\ss E$, then the graph $T\sm E'$ is a forest with some trees $T_i=(V_i,E_i)$, and $\{V_i\}$ is a partition of the set  $V$. This partition we denote by $\cP_T(E')$. If $E'=\{e\}$, then we put $\cP_T\bigl(\{e\}\bigr)=\cP_T(e)$. Notice also, that $\{E_i\}\cup\{E'\}$ is a partition of the edge set $E$.

\begin{rk}
In what follows we sometimes consider paths and cycles as the corresponding subgraphs. In this sense we use the notations  $V(\g)$, $E(\g)$, $V(C)$, $E(C)$ for the vertex sets and edge sets of a path $\g$ and a cycle $C$, respectively.
\end{rk}

We collect the necessary properties of trees in the following Assertion.

\begin{ass}\label{ass:prop_trees}
Let $T=(V,E)$ be an arbitrary tree.
\begin{enumerate}
\item\label{ass:prop_trees:item:1} If $E'\ss E$ and $\cP_T(E')=\{V_i\}$, then the cardinality of the set $\{V_i\}$ is greater by $1$ than the cardinality of the set  $E'$. Besides, any pair $V_i$, $V_j$ is connected by at most one edge from $E$, and if such an edge does exist, then it lies in $E'$. In particular, if $E'=\{e\}$, then  $\cP_T(e)=\{V_1,V_2\}$, and the sets $V_1$ and $V_2$ are connected by a unique edge of the tree $T$, namely, by the edge $e$.
\item\label{ass:prop_trees:item:2} If $V=W_1\sqcup W_2$, $w_i\in W_i$, then there exists an edge $f\in E\bigl(T[w_1,w_2]\bigr)$, connecting  $W_1$ and $W_2$. And if  $\{W_1,W_2\}=\cP_T(e)$ for some edge $e\in E$, then the edge $f$ is unique and coincides with $e$.
\item\label{ass:prop_trees:item:3} Let $v,\,w\in V$, $v\ne w$, and $e\in E$. Then the graph $T[e\to vw]$ is a tree, if and only if $e\in E\bigl(T[v,w]\bigr)$.
\item\label{ass:prop_trees:item:4} If $\{V_1,V_2\}=\cP_T(e)$ for some $e\in E$, and $v_i\in V_i$, then $e\in E\bigl(T[v_1,v_2]\bigr)$.
\item\label{ass:prop_trees:item:5} Under the assumptions of the previous Item, put $S=T[e\to v_1v_2]$, then $\cP_S(v_1v_2)=\cP_T(e)$.
\end{enumerate}
\end{ass}

\section{Metric Graphs, Minimal Spanning Trees}
\markright{\thesection.~Metric Graphs, Minimal Spanning Trees}
For an arbitrary set $M$ by  $\cD(M)$ we denote the set of all metrics on $M$. There exists a natural partial order on $\cD(M)$, namely, for $\r_1,\,\r_2\in\cD(M)$ we put $\r_1\le\r_2$, if the inequality $\r_1(x,y)\le\r_2(x,y)$ holds for any $x,\,y\in M$.

Let  $G=(V,E)$ be an arbitrary graph, and $\r\in\cD(V)$. Then for any edge $e=vw\in E$ its  \emph{length $\r(e)$} is defined as follows: $\r(e)=\r(v,w)$; the value $\r(G)=\sum_{e\in E}\r(e)$ is called the \emph{length of the graph $G$} (with respect to the metric $\r$).

Now let $T=(V,E)$ be a tree, and $\r\in\cD(V)$. An edge $e\in E$ is called {\em exact\/} (with respect to the metric $\r$), if $\r(e)=\r(V_1,V_2)$, where $\{V_1,V_2\}=\cP_T(e)$.

Let $M$ be an arbitrary set. By $\cT(M)$ we denote the set of all trees with vertex set $M$. Let $\r\in\cD(M)$. The value $\mst(M,\r)=\inf_{T\in\cT(M)}\r(T)$ is referred as the \emph{length of minimal spanning tree}, and a tree $T\in\cT(M)$ such that $\r(T)=\mst(M,\r)<\infty$ is called a \emph{minimal spanning tree on $(M,\r)$}. By $\MST(M,\r)$ we denote the set of all minimal spanning trees on $(M,\r)$. Notice that the set $\MST(M,\r)$ can be empty, and that the length of minimal spanning tree is defined in the case of non-existence of a minimal spanning tree also. Put $\cD_\mst(M)=\bigl\{\r\in\cD(M)\mid\MST(M,\r)\ne\0\bigr\}$. One of the problems discussed in this paper is to describe the metrics belonging to $\cD_\mst(M)$. A metric space $(M,\r)$ is called {\em good}, if  $\mst(M,\r)<\infty$, see~\cite{INT}. Due to definitions, all the metrics from $\cD_\mst(M)$ are good. It is not difficult to see that for any good space $(M,\r)$ the set $M$ is at most countable.

\begin{ass}\label{ass:prop_mst}
Let $(M,\r)$ be an arbitrary metric space, and $T=(M,E)\in\MST(M,\r)$. Then
\begin{enumerate}
\item\label{ass:prop_mst:item:1} for any $v,\,w\in M$ and any $e\in T[v,w]$ the inequality $\r(e)\le\r(v,w)$ is valid\/\rom;
\item\label{ass:prop_mst:item:2} for any $e\in E$, $\{M_1,M_2\}=\cP_T(e)$, we have $\r(e)=\r(M_1,M_2)$, i.e., each edge of a minimal spanning tree is exact\/\rom;
\item\label{ass:prop_mst:item:3} if $M=M'_1\sqcup M'_2$, $v_i\in M'_i$, and $\r(v_1,v_2)=\r(M'_1,M'_2)$, then there exists  $T'=(M,E')\in\MST(M,\r)$ such that $v_1v_2\in E'$.
\item\label{ass:prop_mst:item:4} under the assumptions of the previous Item, if the pair $v_1v_2$ is unique in addition, then $v_1v_2\in E'$ for any $T'=(M,E')\in\MST(M,\r)$, in particular, for $T'=T$.
\end{enumerate}
\end{ass}

\begin{proof}
\eqref{ass:prop_mst:item:1} Due to Assertion~\ref{ass:prop_trees}, Item~\eqref{ass:prop_trees:item:3}, the graph $T[e\to vw]$ is a spanning tree, therefore $\r\bigl(T[e\to vw]\bigr)\ge\r(T)$, and hence $\r(v,w)\ge\r(e)$.

\eqref{ass:prop_mst:item:2} Choose arbitrary $v_i\in M_i$. Then, due to Assertion~\ref{ass:prop_trees}, Item~\eqref{ass:prop_trees:item:4}, $e\in E\bigl(T[v_1,v_2]\bigr)$. It remains to apply Item~\eqref{ass:prop_mst:item:1} of this Assertion.

\eqref{ass:prop_mst:item:3} In accordance with Assertion~\ref{ass:prop_trees}, Item~\eqref{ass:prop_trees:item:2}, there exists an edge $f\in E\bigl(T[v_1,v_2]\bigr)$ connecting $M'_1$ and $M'_2$. But then $\r(v_1,v_2)\le\r(f)$. Due to Item~\eqref{ass:prop_mst:item:1} of this Assertion, the inverse inequality is valid, and hence $\r(f)=\r(v_1,v_2)$. Due to Assertion~\ref{ass:prop_trees}, Item~\eqref{ass:prop_trees:item:3}, the graph $T[f\to v_1v_2]$ is a spanning tree, and, as we have shown just now, $\r\bigl(T[f\to v_1v_2]\bigr)=\r(T)$, therefore $T[f\to v_1v_2]\in\MST(M,\r)$.

\eqref{ass:prop_mst:item:4} Under the notations of the previous Item, $\r(f)=\r(v_1,v_2)$. Since the pair $v_1v_2$ is uniquely defined, then we have $f=v_1v_2$, and hence $v_1v_2\in E$. Since $T$ is an arbitrary one, then we obtain the result required.
\end{proof}

Let  $A$ and $B$ be two non-empty subsets of a metric space $(M,\r)$. We say that the distance between $A$ and $B$ \emph{is attained at $ab$}, if there exist $a\in A$ and $b\in B$ such that $\r(a,b)=\r(A,B)$.

Assertion~\ref{ass:prop_mst} can be used to prove non-existence of minimal spanning trees. Let us give several examples.

\begin{examp}\label{examp:sequence_tent_to_0}
Let $M=\{1/n\}_{n\in\N}\cup\{0\}\ss[0,1]$, and let the distance function be standard: $\r(x,y)=|x-y|$. Show that $\MST(M,\r)=\0$.

Assume the contrary,  and let $T=(M,E)\in\MST(M,\r)$. Notice that for any partition  of $M$ into $M'_1=\{1/k\}_{k\le n}$ and $M'_2=\{0\}\cup\{1/k\}_{k>n}$, $k\in\N$, the pair $\bigl\{1/n,1/(n+1)\bigr\}$ is the unique one which the distance between  $M'_1$ and $M'_2$ is attained at. Due to Assertion~\ref{ass:prop_mst}, Item~\eqref{ass:prop_mst:item:4}, this pair is an edge of the tree  $T$.

Let $e\in E$ be an arbitrary edge going out of  $0$, and $\{M_1,M_2\}=\cP_T(e)$. Assume that $1\in M_1$. Since each  vertex $1/n$ is connected with $1$ by a path in $T$, which does not pass through $e$, then we conclude that $\{1/n\}_{n\in\N}\ss M_1$, and hence, $M_2$ consists of at most one vertex $0$. Since both $M_i$ are not empty, then we have: $M_1=\{1/n\}_{n\in\N}$ and $M_2=\{0\}$.

In accordance with Assertion~\ref{ass:prop_mst}, Item~\eqref{ass:prop_mst:item:2}, the edge $e$ is exact, but $\r(M_1,M_2)=0$, a contradiction.
\end{examp}

\begin{examp}
Let $M=\{1/n\}_{n\in\N}\cup\{-1/n\}_{n\in\N}\ss[-1,1]$, and let the distance function be standard: $\r(x,y)=|x-y|$. Let us show that $\MST(M,\r)=\0$.

Assume the contrary again, i.e\. let there exist $T=(M,E)\in\MST(M,\r)$. We put $M_1=\{1/n\}_{n\in\N}$, $M_2=\{-1/n\}_{n\in\N}$, then $M=M_1\sqcup M_2$. As in Example~\ref{examp:sequence_tent_to_0}, show that each pair $\bigl\{1/n,1/(n+1)\bigr\}$ is an edge of $T$ and each pair $\bigl\{-1/n,-1/(n+1)\bigr\}$ is an edge of $T$ also. Therefore, if $e$ is an arbitrary edge of the tree $T$ connecting $M_1$ and $M_2$, then $\cP_T(e)=\{M_1,M_2\}$. Again the edge $e$ must be exact, but $\r(M_1,M_2)=0$, a contradiction.
\end{examp}

The following result generalizes the both above Examples.

\begin{ass}
Let  $(M,\r)$ be a metric space, and assume that the set $M$ is infinite. Assume that there exist $x,\,y\in M$ such that $\r(x,y)=\mst(M,\r)$, then $\MST(M,\r)=\0$.
\end{ass}

\begin{proof}
Assume the contrary, i.e., let there exist $T=(M,E)\in\MST(M,\r)$. Consider the path $\g=T[x,y]$, then $\r(T)>\r(\g)$, because $M$ is infinite. On the other hand, $\r(\g)\ge\r(x,y)=\mst(M,\r)=\r(T)$, a contradiction.
\end{proof}

By $G_\min(M,\r)$ we denote the graph $(M,E)$, where $vw\in E$, if and only if there exists a partition $\{M',M''\}$ of the set $M$ into non-empty sets, such that the distance between those subsets is attained at $vw$.

\begin{ass}\label{ass:MST_and_min_graph}
Assume that $\MST(M,\r)\ne\0$, then the graph $G_\min(M,\r)$ is connected.
\end{ass}

\begin{proof}
Indeed, let $T\in\MST(M,\r)$, then, due to Assertion~\ref{ass:prop_mst}, Item~\eqref{ass:prop_mst:item:2}, $T\ss G_\min(M,\r)$.
\end{proof}

\begin{rk}
The inverse statement to Assertion~\ref{ass:MST_and_min_graph} does not hold. It is very easy to construct an example of a non good set. It suffices to consider a countable set $M$ endowed with the distance function $\rho$ which is equal to $1$ at each pair of distinct points from $M$. Then for any point $m\in M$ the distance between  $\{m\}$ and $M\setminus\{m\}$ is equal to one, therefore the graph $G_\min(M,\r)$ is the complete graph on $M$. But the metric space $(M,\rho)$ is not good, hence  $\MST(M,\r)=\0$.

Let us also construct an example of a good space with the same property. Put $N=\{1/n\}_{n\in\N}\cup\{0\}\ss[0,1]$, and let $M=\{x\}\sqcup N$. Define on $M$ the following metric $\r$: we put $\r$ be equal to the standard Euclidean metric on $N$, and put $\r(x,0)=\r(x,1/n)=1$ for all $1/n\in N$. Then  $\r(x,N)=1$, and hence all the pairs $\{x,t\}$, $t\in N$, are edges of the graph $G_\min(M,\r)$, therefore this graph is connected.

Show that $\MST(M,\r)=\0$. Assume the contrary, i.e., there exists $T=(M,E)\in\MST(M,\r)$. Put $M'_1=\{1/n\}_{n\in\N}$ and $M'_2=\{0,x\}$. Reasonings similar to the ones given in the discussion of Example~\ref{examp:sequence_tent_to_0}, show that if some edge $e\in E$ joins $M'_1$ and $M'_2$, and $\cP_T(e)=\{M_1,M_2\}$, $1\in M_1$, then $M_1\sp M'_1$. Therefore, either $M_2=\{0,x\}$, either $M_2=\{0\}$, or $M_2=\{x\}$. But the cases$M_2=\{0,x\}$ and $M_2=\{0\}$ can not appear, because in those cases $\r(M_1,M_2)=0$, that contradicts to Assertion~\ref{ass:prop_mst}, Item~\eqref{ass:prop_mst:item:2}. In particular, $0x\not\in E$. Thus, it remains to consider the case, when $M_2=\{x\}$, $M_1=N$, and $x$ is connected by a unique edge with some $1/n$ (namely, by the edge $e$).

Now consider $M''_1=\{1/n\}_{n\in\N}\cup\{x\}$ and $M''_2=\{0\}$, and let an edge $e'\in E$ connects $M''_1$ and $M''_2$. Since any two points from $M''_1$ are joined by a path in $T$, which does not pass through $e'$, then we have $\cP_T(e')=\{M''_1,M''_2\}$. But  $\r(M''_1,M''_2)=0$, and we obtain a contradiction with Assertion~\ref{ass:prop_mst}, Item~\eqref{ass:prop_mst:item:2} again. Thus, $\MST(M,\r)=\0$.
\end{rk}

\begin{ass}
If $\MST(M,\r)\ne\0$, then, for any partition $\{M'_1,M'_2\}$ of the set $M$ such that $\r(M'_1,M'_2)>0$, the distance between  $M'_1$ and $M'_2$ is attained.
\end{ass}

\begin{proof}
Assume the contrary, i.e., let there exist $T=(M,E)\in\MST(M,\r)$, then the lengths of all the edges of the tree $T$ connecting $M'_1$ and $M'_2$ are greater than $\r(M'_1,M'_2)$. Since the length of $T$ is finite, then the number of such edges is finite. Denote these edges by $e_1,\ldots,e_k$.

Due to the definition of the distance between $M'_1$ and $M'_2$, there exist $v_j\in M'_j$, $j=1,\;2$, such that $\r(v_1,v_2)<\r(e_i)$ for all $i=1,\ldots,k$. In accordance with Assertion~\ref{ass:prop_trees}, Item~\eqref{ass:prop_trees:item:2}, some edge $e_i$ is contained in $E\bigl(T[v_1,v_2]\bigr)$. Item~\eqref{ass:prop_trees:item:3} of the same Assertion implies that the graph $T[e_i\to v_1v_2]$ is a spanning tree, but $\r\bigl(T[e_i\to v_1v_2]\bigr)<\r(T)$, a contradiction.
\end{proof}

\begin{thm}\label{thm:crit_mst}
Let $T=(M,E)$ be an arbitrary tree and $\r\in\cD(M)$. Assume that $\r(T)<\infty$. Then $T\in\MST(M,\r)$, if and only if all the edges of the tree $T$ are exact.
\end{thm}

\begin{proof}
At first, let  $T\in\MST(M,\r)$. Then all the edges of the tree $T$ are exact in accordance with Item~\eqref{ass:prop_mst:item:2} of Assertion~\ref{ass:prop_mst}.

Now let all the edges of the tree $T$ be exact. Prove that $T\in\MST(M,\r)$. Assume the contrary, i.e.,  $T\not\in\MST(M,\r)$. Then there exists a tree $T'=(M,E')$ such that $\r(T')<\r(T)$. Put $\e=\r(T)-\r(T')$ and notice that $\e>0$.

Since $T$ and $T'$ are distinct trees with the same vertex set, then there exists $e_1\in E\sm E'$. Put $\cP_T(e_1)=\{M_1,M_2\}$, and let $e_1=v_1v_2$, $v_i\in M_i$. Due to Item~\eqref{ass:prop_trees:item:2} of Assertion~\ref{ass:prop_trees}, there exists an edge $e'_1\in E\bigl(T'[v_1,v_2]\bigr)$ connecting $M_1$ and $M_2$. Since $e'_1\ne e_1$, we have $e'_1\not\in E$ in accordance with the same Item of the same Assertion. Due to the assumptions, the edge $e_1$ is exact, so $\r(e_1)=\r(M_1,M_2)\le\r(e'_1)$. Due to Item~\eqref{ass:prop_trees:item:3} of Assertion~\ref{ass:prop_trees},  the graph $T'_1=T'[e'_1\to e_1]$ is a tree. Since $\r(e_1)\le\r(e'_1)$, we have $\r(T'_1)\le\r(T')$. Notice that the tree $T'_1$ contains at least one edge of the tree $T$.

Since $\r(T'_1)<\r(T)$, then we have $T'_1\ne T$, and so we can repeat the above procedure. Show that the tree $T'_m$ obtained at the $m$th step contains at least  $m$ edges of the tree $T$. To do that it suffices to verify that at each step we chose a new edge of the tree $T$ and do not remove from the corresponding $T'_i$ the edges of $T$ which have been added at the previous steps.

Assume the contrary, i.e., for some $i$ the above statement does not hold. Assume that $i$ is the first such number, hence   $e_1,\ldots,e_{i-1}$ are distinct edges of the tree $T$ belonging to the tree $T'_{i-1}$. Notice that the edge $e_i$ can not coincide with any $e_j$, $j<i$, because $e_i$ does not belong to the tree $T'_{i-1}$ by the construction. It remains to show that the edge chosen to change the edge $e'_i$ is not contained among the edges $e_j$, $j<i$. Since $e'_i$ connects the sets of the partition $\cP_T(e_i)$, and since those sets are connected by a unique edge of the tree $T$, namely, by the edge $e_i$ (see  Item~\eqref{ass:prop_trees:item:1} of Assertion~\ref{ass:prop_trees}), then we conclude: $e'_i\ne e_j$, $j<i$, that is required.

Thus, at the $m$th step of  the procedure described above  we construct a tree $T'_m$ containing at least $m$ edges of the tree  $T$ and such that $\r(T)-\r(T'_m)\ge\e$. It is clear that the value $\r(T)-\r(T'_m)$ does not exceed the sum of the lengths of the remaining edges of the tree $T$. On the other hand, since the length of the tree $T$ is finite, there exists a positive integer $N$ such that for any $m>N$ the sum of lengths of the remaining edges of the tree $T$ is less than $\e$. The contradiction obtained completes the proof.
\end{proof}

\begin{ass}\label{ass:Inf_degree}
Let $T=(M,E)\in\MST(M,\r)$ and $v\in M$, then $\deg v=\infty$, if and only if $\r\bigl(v,M\sm\{v\}\bigr)=0$.
\end{ass}

\begin{proof}
At first, assume that $\deg v=\infty$. If $\r\bigl(v,M\sm\{v\}\bigr)>0$, then $\r(T)=\infty$, a contradiction.

Now let $\r\bigl(v,M\sm\{v\}\bigr)=0$. Then there exists a sequence $v_1,\,v_2,\ldots$ of vertices $v_i\in M$ such that the number sequence $\r(v,v_i)$  monotony tends to $0$. Prove that $\deg v=\infty$. Assume the contrary, i.e., $\deg v<\infty$.

Let $w_1,\ldots,w_n$ be all the vertices of the tree $T$ adjacent with $v$. Put $d_i=\r(v,w_i)$ and $d=\min\{d_i\}>0$. Then there exists $v_k$ such that $\r(v,v_k)<d$, in particular, the vertex $v_k$ is not adjacent with $v$.

Notice that the path $T[v,v_k]$ contains some edge $vw_i$. Due to Assertion~\ref{ass:prop_trees}, Item~\eqref{ass:prop_trees:item:3}, the graph $T[vw_i\to vv_k]$ is a spanning tree, and the above implies that it is shorter than the tree $T$, a contradiction.
\end{proof}

Assertion~\ref{ass:Inf_degree}  can be generalized as follows.

\begin{ass}
Let $T=(M,E)\in\MST(M,\r)$, and let $M=M_1\sqcup M_2$ be some partition of the set $M$. Then $M_1$ and $M_2$ are connected by infinite number of edges in $T$, if and only if $\r(M_1,M_2)=0$.
\end{ass}

\begin{proof}
If  $M_1$ and $M_2$ are connected in $T$ by an infinite number of edges, but  $\r(M_1,M_2)>0$, then $\rho(T)=\infty$, a contradiction.

Conversely, let $\r(M_1,M_2)=0$, and $M_1$ is connected with $M_2$ in $T$ by a finite number of edges, say by  $e_1,\ldots,e_k$. It is clear that the lengths of these edges are separated from zero, so there exists a pair of points $v_i\in M_i$, $i=1,\;2$, such that $\rho(v_1,v_2)<\rho(e_j)$ for all $j=1,\ldots,k$.  Notice that the path $T[v_1,v_2]$ contains one of the edges  $e_j$. For this $j$ the graph  $T[e_j \to v_1v_2]$ is a spanning tree due to Assertion~\ref{ass:prop_trees}, Item~\eqref{ass:prop_trees:item:3}, and this tree is shorter than the tree $T$ in accordance with the above reasonings, a contradiction.
\end{proof}

\section{Weighted Graphs and Corresponding Metrics}
\markright{\thesection.~Weighted Graphs and Corresponding Metrics}
Let  $G=(M,E)$ be an arbitrary graph. Every function $\om\:E\to\R_+$ taking its values in the set $\R_+$ of non-negative reals is called a  \emph{weight function}. The triplet $G=(M,E,\om)$ is called a  \emph{weighted graph}. The {\em weight $\om(G)$} of the weighted graph $G$ is defined as the sum of the weights $\om(e)$ of all the edges $e\in E$, where the sum means the sum of the corresponding number series (if the corresponding series is divergent, then it is infinite). By $\cTW(M)$ we denote the set of all weighted trees $T=(M,E,\om)$ with \emph{positive\/} weight functions $\om$ such that $\om(T)<\infty$. Each $\r\in\cD(M)$ generates the weight function $\om(vw)=\r(v,w)$ on each tree $T=(M,E)$. If $(M,E)\in\MST(M,\r)$, then we sometimes also write  $(M,E,\r)\in\MST(M,\r)$.

For each weighted tree $T=(M,E,\om)\in\cTW(M)$ we define two metrics $\r^T_1,\,\r^T_\infty\in\cD(M)$ as follows (here we put $\max(\0)=0$):
$$
\r^T_1(v,w)=\sum_{e\in E\left(T[v,w]\right)}\om(e), \qquad
\r^T_\infty(v,w)=\max\Bigl\{\om(e)\mid e\in E\bigl(T[v,w]\bigr)\Bigr\}.
$$
It is easy to see that $\r^T_\infty\le\r^T_1$, therefore the set $\cD_T(M)\ss\cD(M)$ consisting of all the metrics  $\r$ such that  $\r^T_\infty\le\r\le\r^T_1$ is well-defined.

\begin{rk}
Notice that for any $e\in E$ the values $\r(e)$ are the same for all $\r\in\cD_T(M)$.
\end{rk}

\begin{thm}\label{thm:crit_mst_infty}
Let  $(M,\r)$ be a metric space and  $T=(M,E,\r)\in\cTW(M)$. Then $T\in\MST(M,\r)$, if and only if  $\r^T_\infty\le\r$.
\end{thm}

\begin{proof}
Let $T\in\MST(M,\r)$. Consider arbitrary $x,\,y\in M$. If  $x=y$ or $xy\in E$, then  $\r^T_\infty(x,y)=\r(x,y)$. If $xy\not\in E$, then, due to Assertion~\ref{ass:prop_trees}, for any  $e\in E\bigl(T[x,y]\bigr)$ we have  $\r(e)\le\r(x,y)$, and hence, $\r^T_\infty(x,y)\le\r(x,y)$.

Conversely, assume that $\r^T_\infty(x,y)\le\r(x,y)$ for any $x,\,y\in V$. Consider an arbitrary edge $e=vw\in T$ and put $\cP(e)=\{V_1,V_2\}$. Then for any $v_i\in V_i$ we have $\r(v_1,v_2)\ge\r^T_\infty(v_1,v_2)\ge\r(v,w)=\r(e)$, where the first inequality is valid due to our assumptions, and the second one holds because the edge $vw$ belongs to the path $T[v_1,v_2]$.  Therefore, $e$ is exact. It remains to apply Theorem~\ref{thm:crit_mst}.
\end{proof}

\begin{rk}
The triangle inequality implies immediately that $\r\le\r^T_1$ for any $T=(M,E,\r)$.
\end{rk}

\begin{cor}\label{cor:MSTctirMax1}
Let $M$ be an arbitrary set. Then
$$
\cD_\mst(M)=\bigcup_{T\in\cTW(M)}\cD_T(M).
$$
\end{cor}

\begin{proof}
Let $\r\in\cD_\mst(M)$ and $T=(M,E,\r)\in\MST(M,\r)$. Then, due to Theorem~\ref{thm:crit_mst_infty}, $\r\ge\r^T_\infty$. Since we always have $\r\le\r^T_1$,  then $\r\in\cD_T(M)$, and hence, the left hand side of the equality we are proving is contained in its right hand side.

Conversely, let $\r\in\cD_T(M)$ for some weighted tree $T=(M,E,\om)\in\cTW(M)$. Then $\r|_E=\om$, and hence, the weighted tree $S=(M,E,\r)$ coincides with $T$, so $\r^T_\infty=\r^S_\infty$. Therefore, $\r\ge\r^S_\infty$ and, due to Theorem~\ref{thm:crit_mst_infty}, we have $T\in\MST(M,\r)$, so $\r\in\cD_\mst(M)$.
\end{proof}

\begin{examp}\label{examp:any_tree}
Let $(M,E)$ be an arbitrary tree with a countable set of edges $E=\{e_n\}_{n\in\N}$ enumerated in an arbitrary way, and $\{\om_n\}_{n\in\N}$ be a positive number sequence such that $\sum_{n=1}^\infty\om_n<\infty$. Put $\om(e_n)=\om_n$, $n\in\N$, and let $T=(M,E,\om)$ be the corresponding weighted tree. Since $\r^T_\infty\le\r^T_1$, then Theorem~\ref{thm:crit_mst_infty} implies that $T\in\MST(M,\r^T_1)$. Notice that the same construction can be applied to any finite tree too. Thus, \emph{any tree with at most countable set of edges is isomorphic to some minimal spanning tree}.
\end{examp}

\section{Sufficient Condition of Minimal Spanning Tree Existence}
\markright{\thesection.~Sufficient Condition of Minimal Spanning Tree Existence}
We need the following Lemma.

\begin{lem}\label{lem:e-partit}
Let  $T=(M,E)$ be a tree, $\r\in\cD(M)$, and $f\in E$ be an exact edge with respect to the metric $\r$. Let  $e\in E$, $e\ne f$,  $\cP_T(e)=\{M_1,M_2\}$, and let  $\r(M_1,M_2)=\r(m_1,m_2)$ for some points $m_i\in M_i$. Then $f$ is an exact edge of the tree $S=T[e\to m_1m_2]$ too.
\end{lem}

\begin{proof}
Let  $G_i=(V_i,E_i)$, $i=1,\,2,\,3$, be the trees of the forest $T\sm\{e,f\}$. Without loss of generality assume that the edge $e$ connects the sets $V_1$ and $V_2$, and the edge $f$ connects the sets $V_2$ and $V_3$. Then $M_1=V_1$, $M_2=V_2\cup V_3$, $\cP_T(f)=\{V_1\cup V_2,\,V_3\}$. There are two possibilities:

{\bf (1)} $m_2\in V_2$. Then $\cP_T(f)=\cP_S(f)$, and hence, the edge $f$ remains exact.

{\bf (2)} $m_2\in V_3$. Then $\cP_S(f)=\{V_1\cup V_3,V_2\}$.  We need to show that for arbitrary vertices  $x\in V_1\cup V_3$ and $y\in V_2$ the inequality $\r(x,y)\ge\r(f)$ is valid. Indeed, let  $x\in V_3$, then $\r(x,y)\ge\r(V_3,V_2)\ge\r(V_3,V_1\cup V_2)=\r(f)$. Now, let $x\in V_1$. Then
$$
\r(x,y)\ge\r(M_1,M_2)=\r(m_1,m_2)\ge\r(V_1\cup V_2,V_3)=\r(f).
$$
\end{proof}

\begin{thm}\label{thm:distance}
Let $(M,\r)$ be a good metric space. Assume that for any partition of the set $M$ into non-empty subsets $M_1$ and $M_2$ the distance between those subsets is attained. Then a minimal spanning tree on $M$ does exist.
\end{thm}

\begin{proof}
Since $M$ is good, then there exists a tree $G=(M,E)$ of a finite length. Enumerate the edges of the tree $G$ in an arbitrary way and put $E=\{e_1,\dots\}$. Construct trees $G_i$, $i=1,\dots$, in the following way.

If the edge $e_1$ is exact, then we put $e'_1=e_1$ and $G_1=G$. If the edge $e_1$ is not exact, then consider the partition  $\cP_G(e_1)=\{M_1,M_2\}$, find a pair of points $m_j\in M_j$, $j=1,\,2$, such that the distance between  $M_1$ and $M_2$ is attained at it, and put $G_1=G[e_1\to m_1m_2]$. Since $\r(m_1m_2)<\r(e_1)$, then $G_1$ is shorter than $G$. Put $e_1'=m_1m_2$. By definition of the points $m_i$, the edge $e'_1$ is exact. Thus, $G_1$ is shorter than $G$ and its first edge  $e'_1$ is exact.

Assume that a tree $G_{i-1}$ on $M$ is already constructed, its edge set is $\{e'_1,\ldots,e'_{i-1},e_i,\ldots\}$, and the  edges $e'_1,\ldots,e'_{i-1}$ are exact.  Proceed in the same way for the edge $e_i$. Namely, if the edge $e_i$ is exact, then we put $e'_i=e_i$ and $G_i=G_{i-1}$. If the edge $e_i$ is not exact, then consider the partition $\cP_{G_{i-1}}(e_i)=\{M_1,M_2\}$, find a pair of point $m_j\in M_j$, $j=1,\,2$, such that the distance between $M_1$ and $M_2$ is attained at it, and put $G_i=G_{i-1}[e_i\to m_1m_2]$. Since $\r(m_1m_2)<\r(e_i)$, then $G_i$ is shorter than $G_{i-1}$. Put $e_i'=m_1m_2$. Due to definition of the points $m_i$, the edge $e'_i$ is exact. Besides, due to Lemma~\ref{lem:e-partit}, the edges $e'_1,\ldots,e'_{i-1}$ are exact in $G_i$ too. Thus, we obtain the tree $G_i$, such that its first $i$ edges $e'_1,\ldots,e'_i$ are exact, and whose length does not exceed the length of $G$.

Notice that at each step, say the $i$th step, of the procedure described above we do not change the edges $e'_p$, $p<i$, therefore the edges $e'_1,\ldots,e'_p$ are contained in any tree $G_i$, $i\ge p$.

Thus, for each $i$ we defined a pair $e'_i$ of points from the set $M$. Put $E'=\{e'_i\}$ and consider the graph $G'=(M,E')$. Let us show that $G'$ is a tree. We do it in two steps.

(1) The \emph{graph $G'$ does not contain cycles}. Indeed, if the graph $G'$ contains a cycle, then this cycle consists of a finite set of edges $e'_{i_1},\ldots,e'_{i_k}$, and in accordance with the above remark all those edges are contained in the tree $G_j$, where $j=\max_p i_p$, a contradiction.

(2) The \emph{graph $G'$ is connected}. Assume the contrary, and let $M=M_1\sqcup M_2$ be a partition of the set $M$ such that $M_1$ and $M_2$ are not connected by edges from $E'$. Due to our assumptions, the distance between $M_1$ and $M_2$ is attained, and hence, it is positive. Since the graph $G$ is connected, then the sets $M_1$ and $M_2$ are connected by some edges from $E$. Further, since the distance between $M_1$ and $M_2$ is positive, then the set of such edges is finite  (otherwise the length of the tree $G$ is infinite). Let $e_{i_1},\ldots,e_{i_k}$ be those edges. Put $j=\max_p i_p$ and consider the tree $G_j$. Since the graph $G_j$ is connected, then there exists an edge $f\in E(G_j)$ connecting $M_1$ and $M_2$. Recall that $E(G_j)=\{e'_1,\ldots,e'_j,\,e_{j+1},\ldots\}$. Since $\{e_{i_1},\ldots,e_{i_k}\}$ is the set of all the edges from $G$ connecting $M_1$ and $M_2$, then the edge $f$ can not be equal to any of $e_p$, $p>j$, therefore $f$ coincides with some $e'_q$, $q\le j$, and hence, it belongs to $E'$, a contradiction.

Thus,  $G'$ is a tree, its length is finite (because it does not exceed the length of the tree $G$), and each edge of $G'$ is exact. Theorem~\ref{thm:crit_mst} implies that $G'$ is a minimal spanning tree, the proof is completed.
\end{proof}

\begin{rk}
The sufficient condition of Theorem~\ref{thm:distance} is not necessary. As an example we consider a metric space $M$ consisting of all the vertices of the weighted tree $T$, whose edges $e_k$ have weights $1/k^2$, $k=1,\ldots$,  and meet at the unique common vertex $m$, and all the remaining vertices have degree $1$. Define the distance function to be equal to the weights of the paths connecting the corresponding vertices in the tree $T$. Then $T$ is a minimal spanning tree, but the distance between $M_1=\{m\}$ and $M_2=M\sm\{m\}$ is equal to zero and is not attained.
\end{rk}

\section{Locally Minimal Spanning Trees}
\markright{\thesection.~Locally Minimal Spanning Trees}
In the present section we generalize the concept of minimal spanning tree as follows. Let $(M,\rho)$ be a metric space. A tree  $T=(M,E)$ is called a  {\it locally minimal spanning tree}, if for any pair of its vertices $m$ and $m'$ all the edges of the path   $T[m,m']$ are not longer than the distance between $m$ and $m'$. The latter means that the tree can not be ``locally'' shorten be adding a short edge and discarding a longer one from the cycle appeared. Assertion~\ref{ass:prop_mst}, Item~\ref{ass:prop_mst:item:1}, implies that any minimal spanning tree is locally minimal.

\begin{thm}\label{thm:exact}
Let $(M,\rho)$ be a metric space. A tree $T=(M,E)$ is a locally minimal spanning tree, if and only if all its edges are exact.
\end{thm}

\begin{proof}
Let $T$ be a locally minimal spanning tree, $e=m_1m_2$ be its edge, $\cP_T(M)=\{M_1,M_2\}$, and $m_i\in M_i$, $i=1,\;2$. Take any pair of points $v_i\in M_i$, $i=1,\,2$. Due to Assertion~\ref{ass:prop_trees}, Item~\ref{ass:prop_trees:item:4}, the path $T[v_1,v_2]$ contains the edge $e$, and hence, due to locally minimality definition, $\rho(m_1,m_2)\le\rho(v_1,v_2)$, so $\rho(m_1,m_2)=\rho(M_1,M_2)$ as it is required.

Conversely, let all the edges of a tree $T$ be exact. Take an arbitrary pair of vertices $m_1$ and $m_2$ from $M$, consider the path $T[m_1,m_2]$, and let $e\in E\bigl(T[m_1,m_2]\bigr)$. Notice that the vertices $m_1$ and $m_2$ lie in different components of the partition $\cP_T(e)$, so $\r(e)\le\r(m_1,m_2)$, that means locally minimality of the tree $T$ because of arbitrariness  of the choice of $m_i$.
\end{proof}

Theorems~\ref{thm:crit_mst} and~\ref{thm:exact} imply the following result.

\begin{cor}\label{cor:lmst=mst}
A locally minimal spanning tree of a finite length is a minimal spanning tree.
\end{cor}

\section{Acknowledgements}
\markright{\thesection.~Acknowledgements}
The authors are sincerely appreciate to Academician A.\,T.~Fomenko for his kind attention to their work. The work is partly supported by RFBR, Project~13--01--00664a, and also by the President of RF Program supporting Leading Scientific Schools of RF, Project~NSh--581.2014.1.

\end{document}